\documentclass[letterpaper, 10 pt, conference]{ieeeconf}  

\usepackage[skip=2pt,font=footnotesize]{subcaption}

\IEEEoverridecommandlockouts                              

\usepackage[top=0.75in, bottom=0.75in, left=0.75in, right=0.75in]{geometry}
\usepackage[T1]{fontenc}
\usepackage[utf8]{inputenc}
\usepackage{bm}
\usepackage[unicode=true,
 bookmarks=false, 
 breaklinks=false,pdfborder={0 0 1},colorlinks=false]
 {hyperref}
\hypersetup{
 colorlinks,citecolor=blue,filecolor=blue,linkcolor=blue,urlcolor=blue}
\usepackage{xcolor,colortbl}
\definecolor{Gray}{gray}{0.85}
\usepackage{amsmath}
\usepackage{amssymb} 
\makeatletter

\usepackage{cite}
\usepackage{comment}
\usepackage{booktabs,mathtools}
\usepackage{graphicx}
\usepackage{algorithm}
\usepackage{algorithmic}
\usepackage{multirow}
\usepackage{dsfont}
\usepackage{color}
\usepackage{float}
\usepackage[capitalize]{cleveref}
\crefname{assumption}{Assumption}{Assumptions}
\usepackage{url}
\usepackage{tcolorbox}
\usepackage{changepage}

\allowdisplaybreaks

\newtheorem{theorem}{Theorem}
\newtheorem{assumption}{Assumption}
\newtheorem{definition}{Definition}

\newtheorem{remark}{Remark}
\DeclareMathOperator{\ind}{\mathds{1}}  

\newcommand{\xb}{\mathbf{x}}
\newcommand{\yb}{\mathbf{y}}

\newcommand{\bb}{\bm{b}}

\newcommand{\cA}{\mathcal{A}}
\newcommand{\cB}{\mathcal{B}}

\newcommand{\cD}{\mathcal{D}}

\newcommand{\cF}{\mathcal{F}}

\newcommand{\cK}{\mathcal{K}}

\newcommand{\cM}{\mathcal{M}}

\newcommand{\cP}{\mathcal{P}}

\newcommand{\cX}{\mathcal{X}}

\newcommand{\dW}{d_{W}} 
\newcommand{\E}{\mathbb{E}} 
\newcommand{\PP}{\mathbb{P}} 
\newcommand{\hPN}{\hat{P}_N} 
\newcommand{\R}{\mathbb{R}} 

\title{\LARGE \bf
Decision-Dependent Distributionally Robust Optimization with Application to Dynamic Pricing 
}

\author{Chengrui Qu, Huiwen Jia, and Pengcheng You
\thanks{C. Qu and P. You are with the College of Engineering, Peking University, Beijing, China. P. You is the corresponding author.
        {\tt\small qcr2021@stu.pku.edu.cn, pcyou@pku.edu.cn}}%
\thanks{H. Jia is with the Department of Industrial Engineering and Operations Research, University of California, Berkeley, CA, USA.
        {\tt\small huiwenj@berkeley.edu}}%
}

\begin{document}

\maketitle
\thispagestyle{empty}
\pagestyle{empty}

\begin{abstract}
We consider decision-making problems under decision-dependent uncertainty (DDU), where the distribution of uncertain parameters depends on the decision variables and is only observable through a finite offline dataset. To address this challenge, we formulate a decision-dependent distributionally robust optimization (DD-DRO) problem, and leverage multivariate interpolation techniques along with the Wasserstein metric to construct decision-dependent nominal distributions (thereby decision-dependent ambiguity sets) based on the offline data. We show that the resulting ambiguity sets provide a finite-sample, high-probability guarantee that the true decision-dependent distribution is contained within them. Furthermore, we establish key properties of the DD-DRO framework, including a non-asymptotic out-of-sample performance guarantee, an optimality gap bound, and a tractable reformulation. The practical effectiveness of our approach is demonstrated through numerical experiments on a dynamic pricing problem with nonstationary demand, where the DD-DRO solution produces pricing strategies with guaranteed expected revenue.

\end{abstract}

\section{Introduction}

Many real-world decision-making problems involve uncertainty, where the underlying distributions are unobservable and only accessible through limited historical data. Two widely used approaches to handle such uncertainty are Stochastic Programming (SP) and Distributionally Robust Optimization (DRO). Traditional SP optimizes the expected outcome under a fixed distribution, but can suffer from degraded performance when deployed in practice due to overfitting the training data, a phenomenon known as the optimizer's curse \cite{optimizer-curse}. To mitigate this, \emph{distributionally robust optimization} (DRO) has been proposed as a more robust alternative \cite{bental2009,bertsimas2004,kuhn2018}. Given historical data and a robustness level, DRO constructs an \emph{ambiguity set} of plausible distributions and optimizes for the worst-case scenario within this set, leading to a $\min$-$\max$ formulation. DRO has been shown, both theoretically and empirically, to effectively reduce out-of-sample risk and is often more tractable than the original stochastic optimization problem, even when the true data-generating distribution is known \cite{delage2010,goh2010,wiesemann2014,Breton1995}.

\subsection{Decision-Dependent Uncertainty}\label{subsec: ddu}

In many practical problems, a further complication arises: the distribution of the uncertain parameters depends on the decisions themselves—a phenomenon referred to as \emph{endogenous uncertainty}, or more formally, \emph{decision-dependent uncertainty} (DDU)\cite{Goel2006,Dupacova2006,Hellemo2018,yu2022multistage}. This arises in diverse domains including financial market modeling \cite{kurzmotolese2001}, 
robust network design \cite{viswanath2004}, stochastic traffic assignment \cite{hu2006}, pricing \cite{jia2024}, and resource management \cite{tsur2004}. These problems can often be formulated as the following decision-dependent optimization (DDO) problem:
\begin{equation}\label{eq:so}
    \min_{\xb\in\cX}\E_{Q(\xb)}[h(\xb,\xi)], \tag{DDO}
\end{equation}
where $\xb\in \R^d$ is the decision variable, $\cX$ denotes the feasible region,  $\xi \in \R^k$ is the random parameters following a decision-dependent distribution $Q(\xb)$, and $h(\cdot,\cdot): \R^d\times\R^k\mapsto\overline{\R}$ represents the objective function. Here, $\overline{\R}=\R\cup\{\pm \infty\}$ denotes the extended real number.
Since $Q(\xb)$ is unobservable in practice, directly solving \eqref{eq:so} is typically infeasible. Traditional SP and DRO approaches are inadequate in this setting due to the coupled dependency between the inner-layer model uncertainties and the outer-layer optimization decisions. While \cite{dd-saddle} develops a primal–dual method to compute fixed points for DDO in competitive settings, it does not incorporate mechanisms for out-of-sample robustness.

This limitation has motivated a growing body of research that seeks to address DDO using DRO-based approaches \cite{luo2020,noyan2022,zhang2016,royset2017,fonseca2023}, primarily through the construction of decision-dependent ambiguity sets. For example,
\cite{luo2020} shows that DRO with decision-dependent ambiguity sets can be reformulated as $\min$-$\min$ programs with infinitely many constraints, solvable via cutting-surface methods. \cite{noyan2022} proposes tractable reformulations based on the Earth Mover’s Distance. \cite{fonseca2023} studies a standard DRO setting but treats the objective function as a decision-dependent random variable, which can be interpreted as a special case of DRO for DDO. Despite these advances, a core challenge remains: \emph{how to systematically design decision-dependent ambiguity sets of appropriate size while providing provable out-of-sample guarantees for the resulting solutions}.

\subsection{Decision-Dependent DRO (DD-DRO)}

To address this challenge, we introduce a novel interpolation-based \emph{decision-dependent DRO} (DD-DRO) model that leverages multivariate interpolation and the Wasserstein metric to construct ambiguity sets in a data-driven, decision-aware, and robust manner. Our framework guides practitioners in choosing the size of these ambiguity sets based on available data and robustness requirements, thereby avoiding overly conservative or overly optimistic solutions.

This approach has broad practical relevance. One example is the problem of pricing charging services at electric vehicle charging stations (EVCSs), where demand is uncertain, time-dependent, and responsive to prices - that is, decision-dependent. Existing methods often rely on restrictive assumptions such as known demand functions \cite{stefffen2019}, full knowledge of transportation systems \cite{lai2023}, or specific user behaviors \cite{Alizadeh2017}. The literature on service pricing typically resorts to queueing theory to model the complex system dynamics, which is challenging even in the stationary and static pricing setting \cite{jia2022icml,jia2022neurips}. Despite growing interest, a practical, data-driven pricing framework for EVCSs remains lacking \cite{lai2023}. Our method fills this gap by enabling operators to design pricing strategies directly from historical data and a robustness parameter - without modeling complex system dynamics - while still guaranteeing expected revenue. We demonstrate the effectiveness of our approach in this context in \cref{sec:pricing}.

\subsection{Contributions}

To the best of our knowledge, this is the first tractable framework for decision-dependent DRO that offers non-asymptotic guarantees on both out-of-sample performance and optimality gap. 
Our contributions are summarized as follows:
\begin{itemize}
    \item We study optimization problems under decision-dependent uncertainty, where the true distribution is unobservable and accessible only via finite historical data. We propose a novel method for constructing decision-dependent ambiguity sets using the Wasserstein metric and multivariate interpolation under a specified robustness level.
    \item We provide theoretical guarantees for the out-of-sample performance of the DD-DRO solution and establish bounds on its optimality gap. We also show that DD-DRO can be reformulated as a semi-infinite $\min$-$\min$ optimization problem, which can be solved with a cutting-surface algorithm.
    \item We demonstrate the effectiveness of our method through a dynamic pricing application, supported by numerical experiments.
\end{itemize}

\section{Overview: Decision-Dependent Distributionally Robust Optimization}

As we discussed in Section \ref{subsec: ddu}, directly solving \eqref{eq:so} is infeasible without the knowledge of $Q(\xb)$, and in this work, we consider the following \emph{decision-dependent distributionally robust optimization} (DD-DRO) problem:
\begin{equation}\label{DD-DRO}
    \min_{\xb\in\cX}\max_{P\in\cP(\xb)}\E_{P}[h(\xb,\xi)], \tag{DD-DRO}
\end{equation}
where $\cP(\xb) \subseteq \cM(\Xi, \cB(\Xi))$ is the decision-dependent ambiguity set under the decision $\xb$.  Here, $\cB(\Xi)$ denotes the Borel $\sigma$-field on $\Xi$ and $\cM(\Xi, \cB(\Xi))$ is the set of all probability measures on the measurable space $(\Xi, \cB(\Xi))$. We assume access to a pre-collected offline dataset $\cD=\{(\xb_n,\xi_n)\}_{n\in[N]}$, where $[N]$ denotes $\{1,2,\cdots,N\}$ and each $\xi_n$ are sampled from the true but unknown distribution $Q(\xb_n)$. We introduce a systematic method for constructing $\cP(\xb)$ in Section \ref{sec: ambiguity set} and provide a non-asymptotic guarantee (see Theorem \ref{thm:non-asymptotic}) on
$$
\mathbb{P}\left(Q(\xb)\in \cP(\xb),\ \forall \xb\in\cX\right)\ge 1-\beta,
$$
where $\beta\in(0,1)$ is a prescribed confidence level.

\begin{figure}
    \centering
    \begin{subfigure}[b]{0.48\textwidth}
    \includegraphics[width=\linewidth]{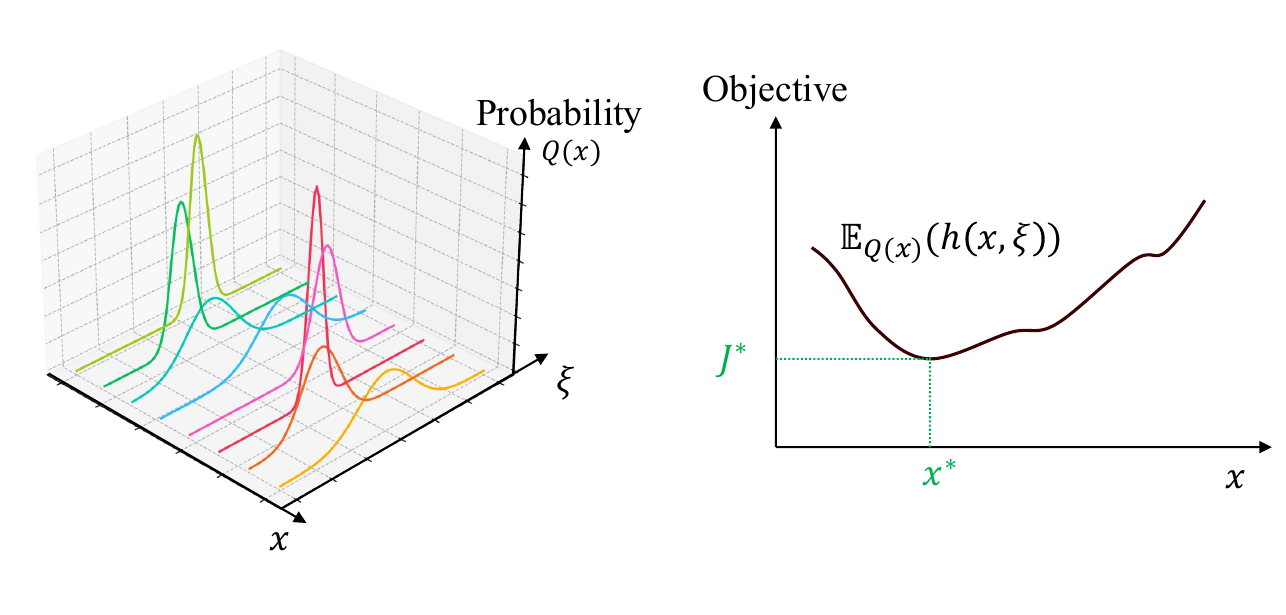}
    \caption{The optimization problem \eqref{eq:so} we aim to solve under DDU.}
    \label{fig:true}
    \end{subfigure}

    \begin{subfigure}[b]{0.48\textwidth}
    \includegraphics[width=\linewidth]{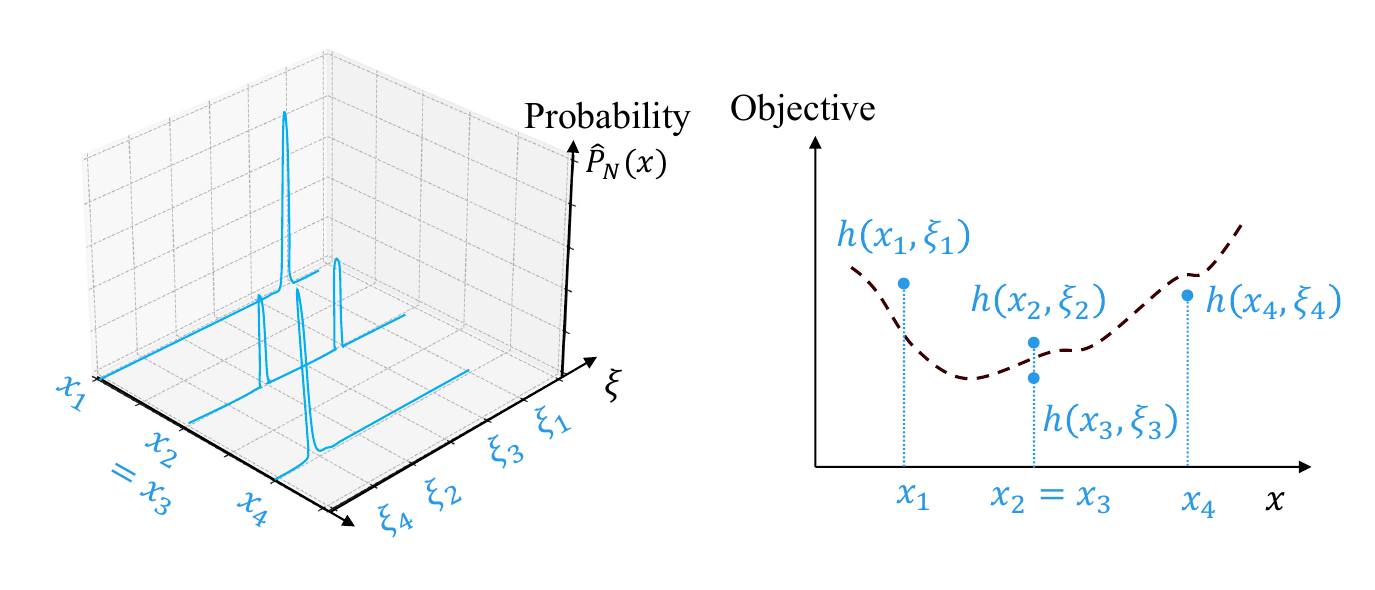}
    \caption{Available data $\cD=\{(\xb_n,\xi_n)\}_{n\in[N]}$ for offline learning.}
    \label{fig:data}
    \end{subfigure}

    \vspace{1mm}

    \begin{subfigure}[b]{0.48\textwidth}
    \includegraphics[width=\linewidth]{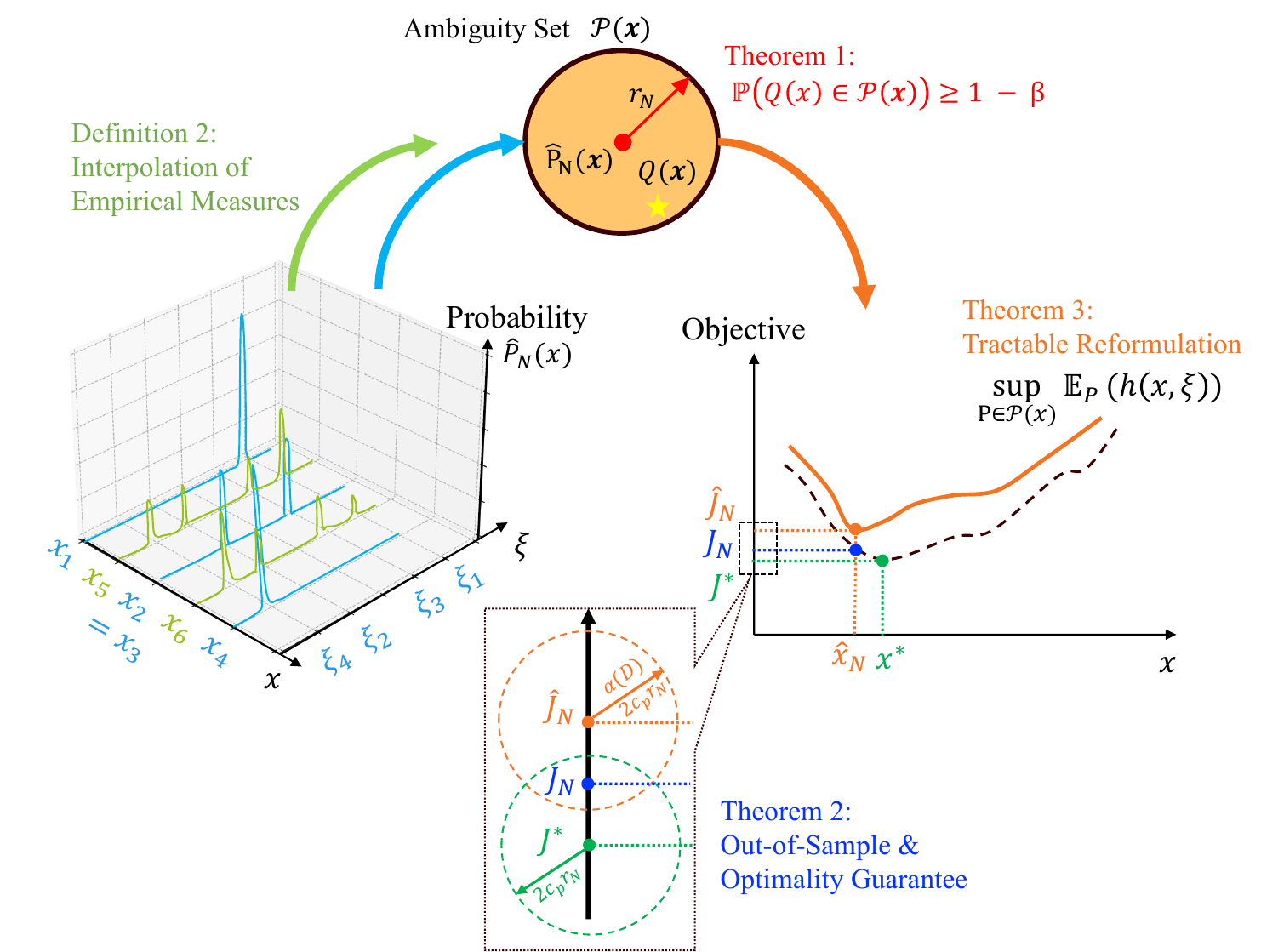}
    \caption{The proposed DD-DRO approach, including interpolation-based nominal distribution construction for $\hat{P}_{N}(x)$, decision-dependent ambiguity set construction for $\mathcal{P}(x)$, and performance guarantee of \eqref{DD-DRO} solution $\hat{\xb}_N$.}
    \label{fig:dd-dro}
    \end{subfigure}

    \caption{Overview of the DDO problem, the available data, and the proposed DD-DRO framework.}
    \label{fig:overview}
\end{figure}

Next, we provide performance guarantees for the solution of \eqref{DD-DRO}, including (i) the out-of-sample performance when it is implemented under the true distribution $Q(\xb)$ and (ii) how close it is to the true optimal objective value.
Let $\hat{J}_N$ denote the optimal value of \eqref{DD-DRO}, and let $\hat{\xb}_N$ be an optimal 
solution. Let 
$$
J_N=\E_{Q(\hat{\xb}_N)}[h(\hat{\xb}_N,\xi)].
$$
denote the true expected out-of-sample performance of the solution $\hat{\xb}_N$, and let $J^\star$ denote the optimal objective value of the original problem \eqref{eq:so}.
Under a shared confidence level $\beta\in(0,1)$ and an error bound $\alpha(\cD,\beta)$ (which depends on the dataset $\cD$ and $\beta$), we establish the following guarantees (see \cref{thm:out-of-sample-performance}):
\begin{itemize}
    \item \textbf{Out-of-sample performance}: The performance of $\hat{\xb}_N$ remains close to its estimated value when deployed in the real environment:
    \begin{equation*}
        \PP\left(|\hat{J}_N-J_N|\le \alpha(\cD,\beta)\right)\ge 1-\beta.
    \end{equation*}
    \item \textbf{Optimality gap}: The true performance of $\hat{\xb}_N$ is close to the best achievable value:
    \begin{equation*}
        \PP \Big( |J_N-J^{\star}|\le \alpha(\cD,\beta) \Big) \ge 1-\beta.
    \end{equation*}
\end{itemize}

To support the practical implementation of DD-DRO, we also provide a tractable reformulation of \eqref{DD-DRO} as a $\min$-$\min$ optimization problem (see \cref{thm:tractable-reformulation}). Figure~\ref{fig:overview} summarizes the DDO problem, the available offline data, our proposed DD-DRO approach, and the associated theoretical guarantees.

\section{Interpolation-based DD-DRO with Wasserstein Metric}\label{sec: ambiguity set}

In this section, we formally introduce our design of the DD-DRO framework using the Wasserstein metric and multivariate interpolation techniques. Specifically, we construct a decision-dependent nominal distribution by interpolating empirical distributions obtained from the offline dataset, and define an ambiguity set centered at this nominal distribution using the Wasserstein metric.

We begin by introducing the Wasserstein metric.
\begin{definition}[$L_1$-Wasserstein Metric \cite{clark}]
    Given a $\sigma$-algebra $\cF$,  for any two probability measures $P_1, P_2\in\cM(\Xi,\cF)$, the $L_1$-Wasserstein distance between them is defined as:
$$\dW(P_1,P_2)\triangleq\inf_{\gamma\in\Gamma(P_1,P_2)}\int_{\Xi\times\Xi}d(s_1,s_2)\gamma(ds_1\times ds_2),
$$
where $\Gamma(P_1,P_2)\triangleq\{\gamma\in\cM(\Xi\times\Xi,\cF\times\cF):\gamma(A\times \Xi)=P_1(A),\gamma(\Xi\times A)=P_2(A),\forall A\in\cF\}$ is the set of all couplings of $P_1$ and $P_2$, and $d(\cdot,\cdot)$ is a $\cF\times\cF$-measurable metric defined on $\Xi$. In our setting, we assume this metric is given by a p-norm $||\cdot||_p$ for $p\ge 1$. 
\end{definition}

We now define the decision-dependent ambiguity sets based on this metric. To make the dependence on the decision variable $\xb$ explicit, we consider a decision-dependent nominal distribution $\hPN(\xb)\in\cM(\Xi,\cB(\Xi))$, and consider ambiguity sets of forms given in \cref{def:ambiguity-set}. 

\begin{definition}\label{def:ambiguity-set}
Given a radius $r_N>0$, a Wasserstein distance based ambiguity set $\cP(\xb)$ for $\hPN(\xb)$ is defined as:
$$
\cP(\xb)\triangleq\{P\in\cM(\Xi,\cB(\Xi))\mid \dW(P,\hPN(\xb))\le r_N\}.
$$
\end{definition}

Directly choosing the nominal distribution as one of the empirical distributions from the offline dataset — as is done in standard DRO approaches \cite{kuhn2018} — is not feasible in our setting, since these distributions do not depend on the decision variable $\xb$. Therefore, we resort to multivariate interpolation methods to interpolate between these empirical measures.

Specifically, for each unique decision point $\xb_{n_i}$ in the offline dataset $\cD$, we construct an empirical distribution $\mu_i$ based on the observed samples of $\xi$. Let $\{\xb_{n_i}\}_{i \in [N_\xb]}$ be the set of distinct decision values, and $\{\xi_{m_i}\}_{i \in [N_\xi]}$ be the set of distinct realizations of $\xi$ in $\cD$, where $N_\xb$ and $N_\xi$ are the numbers of distinct realizations of $\xb$ and $\xi$ in the dataset, respectively. Let $\ind[\cdot]$ denote the indicator function. For each $\xb_{n_i}$, define 
 $N_i\triangleq\sum_{n=1}^N\ind[\xb_n=\xb_{n_i}]$, and the corresponding empirical measure of $\xi$ at $\xb_{n_i}$:
\begin{equation*}
    \mu_i(\xi) \triangleq \frac{1}{N_i} \sum_{\{n\in [N]\mid\xb_n=\xb_{n_i}\}}\delta_{\xi_n}(\xi),
\end{equation*}
where $\delta_{\xi_n}(\cdot)$ denotes the Dirac measure at the sample realization $\xi_n$. We assume $\cX$ is equipped with a metric $d_\cX(\cdot,\cdot)$. We now define an interpolation over these empirical distributions.
\begin{definition}[Interpolation of Empirical Measures]\label{def:interpolation}
    Given empirical measures $\{\mu_i(\cdot)\}_{i\in[N_\xb]}$ corresponding to decisions $\{\xb_{n_i}\}_{i \in [N_\xb]}$, an interpolation of these empirical measures is defined as:
    \begin{equation*}
        \hPN(\xb)\triangleq \sum_{i=1}^{N_\xb}\omega_i(\xb)\mu_i(\cdot),\ \forall \xb\in\cX,
    \end{equation*}
    where the weights $\{\omega_i(\xb)\}_{i\in[N_\xb]}$ satisfy:
    $$
    \left\{
    \begin{aligned}
        &\omega_i(\xb)\ge 0,\ \forall i\in[N_\xb], \forall \xb\in\cX\\
        \sum_{i=1}^{N_\xb} &\omega_i(\xb)=1,\ \forall \xb\in\cX, \\
        &\omega_i(\xb_{n_i})=1,\ \forall i \in[N_\xb].\\
    \end{aligned}
    \right.
    $$
    In addition, we require the interpolation to be Lipschitz continuous at each $\xb_{n_i}$ with respect to the Wasserstein distance. That is, there exists a constant $c_1\ge 0$, such that for any $\xb\in\cX$, 
    \begin{equation*}
        \dW(\hPN(\xb),\hPN(\xb_{n_i}))\le c_1 \inf_{j\in[N_\xb]} d_\cX(\xb,\xb_{n_j}).
    \end{equation*}
\end{definition}
\begin{remark}
Intuitively, the interpolation is defined through a set of decision-dependent weights that vary smoothly with respect to $\xb$ and exactly recover the empirical distributions at each observed decision point. Such an interpolation always exists and can be efficiently calculated. For example, nearest-neighbor interpolation yields a valid construction, where for each $i\in[N_\xb]$,
$$
    \omega_i(\xb)=\ind\left[d_\cX(\xb,\xb_{n_i})=\inf_{j\in[N_\xb]}\{d_\cX(\xb,\xb_{n_j})\}\right],
$$
in which case the Lipschitz constant $c_1=0$. 
\end{remark}
The resulting distribution $\hPN(\xb)$ is supported on the finite set $\{\xi_{m_i}\}_{i \in [N_\xi]}$, and can be rewritten as:
\begin{equation}\label{eq:def-hpn}
    \hPN(\xb)=\sum_{i=1}^{N_\xi}f_i(\xb)\delta_{\xi_{m_i}}(\cdot),
\end{equation}
where
\begin{equation}\label{eq:def-of-f_i}
f_i(\xb)\triangleq\sum_{n=1}^N\ind[\xi_n=\xi_{m_i}]\sum_{j=1}^{N_\xb}\frac{\omega_j(\xb)}{N_j}\ind[\xb_n=\xb_{n_j}].
\end{equation}
It is straightforward to verify that $\sum_{i=1}^{N_\xi}f_i(\xb)=1$. Due to this finite support, using the interpolated measures $\hPN(\xb)$ as the nominal distribution in the ambiguity set enables tractable computation of the Wasserstein distance. Moreover, since $\hPN(\xb)$ is derived from empirical data, measure concentration results can be applied to establish finite-sample guarantees (see Section \ref{sec: theoretical}).
\section{theoretical guarantees}\label{sec: theoretical}

In this section, we formally discuss the coverage property of the designed ambiguity set, the out-of-sample performance guarantee, the optimality guarantee, and the tractability of DD-DRO.
\subsection{Necessary Assumptions}
We first introduce the necessary assumptions for the theoretical analysis.
\begin{assumption}\label{assumption:compact-continuous}
    The sets $\cX$ and $\Xi$ are compact and bounded, and the objective function $h(\cdot,\cdot)$ is bounded on $\cX\times \Xi$. Moreover, for any $\xb\in\cX$, there exists a constant $c_p>0$ such that for any $s_1,s_2\in\Xi$,
    \begin{equation*}
        \left|h(\xb,s_1)-h(\xb,s_2)\right|\le c_p ||s_1-s_2||_p.
    \end{equation*}
\end{assumption}
\cref{assumption:compact-continuous} is a standard regularity condition in DRO literature. In the context of DD-DRO, we further introduce the following assumption to ensure that the problem is well-posed.
\begin{assumption}[Lipschitz Continuity]\label{assumption:continuous}
    There exists a constant $c_2\ge 0$ such that for any $\xb,\yb\in\cX$:
    \begin{equation*}
        \dW\left(Q(\xb),Q(\yb)\right)\le c_2 d_\cX(\xb,\yb).
    \end{equation*}
\end{assumption}
\cref{assumption:continuous} ensures that small changes in the decision variable do not lead to large changes in the true distribution $Q(\xb)$. With these assumptions in place, we now present the main theoretical results.

\subsection{High-Probability Coverage of Ambiguity Set}

As stated before, we construct decision-dependent ambiguity sets of forms in \cref{def:ambiguity-set} for DD-DRO, and choose $\hPN(\xb)$ to satisfy \cref{def:interpolation}. Under a given confidence level $\beta\in(0,1)$, we choose $r_N$ as follows. Let $\cA_r(\xb)\triangleq\{\yb \in\cX\mid d_\cX(\xb,\yb)\le r\}$ denote the ball centered at $\xb$ with radius $r>0$, and let $r_\cD\triangleq\inf\{r>0\mid \cX\subseteq \cup_{i\in [N_\xb]}\cA_r(\xb^i)\}$ be the minimum covering radius of $\cX$. Then $r_N$ is chosen as:
\begin{equation}\label{eq:def-rn}
    r_N= (c_1+c_2)r_\cD+\left(\frac{b(\beta,c_3)}{c_4}\right)^{\frac{1}{k}},
\end{equation}
where $c_1$ and $c_2$ is defined in \cref{def:interpolation} and \cref{assumption:continuous} respectively, $c_3$ and $c_4$ are constants independent of $\cD$, $k$ is the dimension of $\xi$, and $b(\beta,c)\triangleq \inf\{t>0\mid \sum_{i=1}^{N_\xb}\exp\{-tN_i\}<\frac{\beta}{c}\}$ is a sample-dependent term.

We show that with probability at least $1-\beta$, the true distribution is contained in such an ambiguity set.

\begin{theorem}[Coverage of Ambiguity Set]\label{thm:non-asymptotic}
    Under \cref{assumption:compact-continuous,assumption:continuous}, 
     given $\hPN(\xb)$ in \cref{eq:def-hpn} and $r_N$ in \cref{eq:def-rn}, the ambiguity set $\cP(\xb)$ satisfies:
    \begin{equation}
        \mathbb{P}(Q(\xb)\in \cP(\xb),\ \forall \xb\in\cX)\ge 1-\beta.
    \end{equation}
\end{theorem}


\begin{proof}
    By the triangle inequality of the Wasserstein metric \cite[Theorem 2.5]{wasserstein-triangle}, for any $\xb,\xb_{n_i}\in\cX$, we have:
    \begin{align*}
        d_W(\hPN(\xb),Q(\xb))&\le d_W(\hPN(\xb),\hPN(\xb_{n_i}))\\
        &\quad+d_W(\hPN(\xb_{n_i}),Q(\xb_{n_i}))\\
        &\quad+d_W(Q(\xb_{n_i}),Q(\xb)).
    \end{align*}
    By the definition of $r_\cD$, there exists $i\in[N_\xb]$ such that $d_\cX(\xb,\xb_{n_i})\le r_\cD$. Then, by \cref{def:interpolation}:
    \begin{equation*}
        d_W\left(\hPN(\xb),\hPN(\xb_{n_i})\right)\le c_1r_\cD,
    \end{equation*}
    and by \cref{assumption:continuous}:
    \begin{equation*}
        d_W(Q(\xb_{n_i}),Q(\xb))\le c_2r_\cD.
    \end{equation*}
    For the second term, by measure concentration results for the Wasserstein metric \cite[Theorem 2]{fournier}, there exists constants $c_3,c_4>0$ that are independent of the offline dataset $\cD$, such that for any $\beta_i\in(0,\beta)$, with probability at least $1-\beta_i$:
    \begin{equation*}
        d_W\left(\hPN(\xb_{n_i}),Q(\xb_{n_i})\right)\le \left(\frac{\log(c_3/\beta_i)}{c_4N_i}\right)^\frac{1}{k}.
    \end{equation*}
    For $\beta_i\in(0,\beta)$ satisfying that $\sum_{i=1}^{N_\xb}\beta_i=\beta$, by the union bound, for any $i\in[N_\xb]$, with probability at least $1-\beta$, we have the following inequality:
    \begin{equation*}
        d_W\left(\hPN(\xb_{n_i}),Q(\xb_{n_i})\right)\le \max_{i\in[N_\xb]}\left(\frac{\log(c_3/\beta_i)}{c_4N_i}\right)^\frac{1}{k}.
    \end{equation*}
    Optimizing over $\beta_i$, we have:
    \begin{equation*}
        d_W\left(\hPN(\xb_{n_i}),Q(\xb_{n_i})\right)\le \left(\frac{b(\beta,c_3)}{c_4}\right)^{\frac{1}{k}},\ \forall i\in[N_\xb].
    \end{equation*}
    Combining all three parts completes the proof.
\end{proof}
\cref{thm:non-asymptotic} shows that in order for the ambiguity set $\cP(\xb)$ to contain the true distribution, the radius $r_N$ must consist of two components:
\begin{enumerate}
    \item a covering term $(c_1+c_2)r_\cD$ capturing the sparsity of the observed decisions $\xb_{n_i}$.
    \item a statistical term $(\frac{b(\beta,c_3)}{c_4})^{\frac{1}{k}}$ reflecting the sample sufficiency at each $\xb_{n_i}$.
\end{enumerate}

\subsection{Out-of-Sample and Optimality Guarantee}

Recall that $\hat{J}_N$ and $\hat{\xb}_N$ are the optimal objective value and an optimal solution to \eqref{DD-DRO}, and $J_N$ denotes the expected objective value of $\hat{\xb}_N$ under $Q(\hat{\xb}_N)$ and $J^\star$ is the optimal objective value of \eqref{eq:so}.

\begin{theorem}[Out-of-Sample Performance]\label{thm:out-of-sample-performance}
    Under \cref{assumption:compact-continuous,assumption:continuous}, given $\hPN(\xb)$ in \cref{eq:def-hpn} and $r_N$ in \cref{eq:def-rn}, then with probability at least $1-\beta$, we have:
     \begin{itemize}
         \item $J^\star \le J_N\le  \hat{J}_N$,
         \item $\hat{J}_N-2c_p\cdot r_N \le J_N\le \hat{J}_N$,
         \item $J^\star \le J_N\le J^\star+2c_p\cdot r_N$.
     \end{itemize}
\end{theorem}
\begin{proof}
    By \cref{thm:non-asymptotic}, w.p. $1-\beta$, for any $\xb\in\cX$, $Q(\xb)\in\cP(\xb)$, from the definition of \eqref{DD-DRO}, we directly have $J_N\le \hat{J}_N$, and from the optimality of $J^\star$, we also have $J^\star\le J_N$. Define the normalized function $f(\xi)=1/c_p h(\hat{\xb},\xi)$. Under \cref{assumption:compact-continuous}, $f$ is 1-Lipschitz. Let $Q_W(\hat{\xb})$ be the worst-case distribution in the ambiguity set:
    \begin{equation*}
        Q_W(\hat{\xb})=\arg\max_{P\in\cP(\hat{\xb})}\E_P[h(\hat{\xb},\xi].
    \end{equation*}
    Let $||\cdot||_{Lip}$ denote the Lipschitz constant of a function. By the Kantorovich–Rubinstein representation \cite[Section 11.8]{Dudley_2002} for $L_1$ Wasserstein distance, we have:
    \begin{align*}
        |\hat{J}_N-J_N|&=c_p|\E_{Q_W(\hat{\xb})}f-\E_{Q(\hat{\xb})}f|\\
        &\le c_p\sup_{||f||_{Lip}\le 1}|\E_{Q_W(\hat{\xb})} f-\E_{Q(\hat{\xb})} f|\\
        &=c_p\cdot d_W(Q_W(\hat{\xb}),Q(\hat{\xb}))\\
        &\le c_p\cdot d_W(Q_W(\hat{\xb}),\hPN(\hat{\xb}))\\
        &\quad+c_p\cdot d_W(\hPN(\hat{\xb}),Q(\hat{\xb}))\\
        &\le 2c_p\cdot r_N,
    \end{align*}
    where the second equality follows from \cite[Section 11.8]{Dudley_2002}, and the last inequality follows from the construction of the ambiguity set. This yields that
    \begin{equation*}
        \hat{J}_N-2c_p\cdot r_N\le J_N\le \hat{J}_N.
    \end{equation*}
    Similarly, defining $g(\xi)=1/c_ph(\xb^\star,\xi)$ and denoting $\hat{J}^\star=\sup_{P\in\cP(\xb^\star)}\mathbb{E}_{P} [h(\xb^\star,\xi)]$,
    we can show $|\hat{J}^\star-J^\star|\le 2c_p\cdot r_N$. From the optimality of $\hat{J}^\star$, we can remove $|\cdot|$ and obtain $\hat{J}^\star-J^\star\le 2c_p\cdot r_N$. Furthermore, because $\hat{J}_N=\inf_{\xb}\sup_{P\in\cP(\xb)}E_P[h(\xb,\xi)]$, we have:
    \begin{equation*}
        \hat{J}_N\le \hat{J}^\star.
    \end{equation*}
    Therefore, we obtain 
    \begin{equation*}
        \hat{J}_N\le J^\star+ 2c_p\cdot r_N.
    \end{equation*}
    Recall $J^\star\le J_N\le \hat{J}_N$, this further yields
    \begin{equation*}
        J^\star\le J_N\le J^\star+2c_p\cdot r_N.
    \end{equation*}
    This ends the proof.
\end{proof}
\subsection{Tractability of DD-DRO Problem}
\begin{theorem}[Tractable Reformulation]\label{thm:tractable-reformulation}
    Under \cref{assumption:compact-continuous}, given $\hPN(\xb)$ in \cref{eq:def-hpn} and $f_i(\xb)$ in \cref{eq:def-of-f_i}, 
    \eqref{DD-DRO} is equivalent to the following semi-infinite optimization problem:
    \begin{align*}
        \min_{\xb\in\cX,\nu\in\R^{N_\xi+1}}&\ \sum_{i=1}^{N_\xi}f_i(\xb)\nu_i+r_N\cdot\nu_{N_\xi+1} \\
        \mathrm{s.t. }\ &h(\xb,s)-\nu_i-\nu_{N_{\xi}+1}\cdot d(s,\xi_{m_i})\le 0,\\
        &\forall s\in\Xi,\forall i\in[N_\xi],\\
        &\ \nu_1,\cdots,\nu_{N_\xi}\in\R,\ \nu_{N_\xi+1}\ge 0.
    \end{align*}
\end{theorem}
This problem can be solved using a cutting-surface algorithm \cite{luo2020}[Section 4]. The proof of \cref{thm:tractable-reformulation} can be found in \cref{appendix:tractable-reformulation}. In practice, when the objective function is convex or concave, \eqref{DD-DRO} can be further reduced to an optimization problem with finitely many convex constraints, as we will demonstrate in an application to dynamic pricing in \cref{sec:pricing}.

\section{Case Study: Dynamic Pricing}\label{sec:pricing}
In this section, we demonstrate the practical effectiveness of our proposed method through an application to dynamic pricing.

Consider a firm that sets prices $\xb\in\cX\triangleq[0,x_U]^T$ for its goods or services over a finite horizon $T$ in an ex-ante manner. At each time period $t\le T$, the firm observes demand $\xi_t\in[0,\xi_U]$, and earns revenue $x_t\xi_t$ for satisfying it. Rather than assuming a specific demand model, we consider the general case where the demand vector $\xi \in \Xi \triangleq [0, \xi_U]^T$ follows an unknown, decision-dependent joint distribution $Q(\xb)$. This decision-dependent distribution $Q(\xb)$ models the fact that higher prices may depress demand while lower prices may increase it, with the precise relationship unknown. The cumulative revenue is defined as:
\begin{equation}\label{eq:pricing-obj}
    R(\xb,\xi)\triangleq \xb^\mathrm{T}\xi= \sum_{t=1}^T x_t\xi_{t}.
\end{equation}
The firm has access to historical pricing and demand pairs $(\xb_n, \xi_n)$, which may have been generated under different pricing strategies, and aim to set prices to maximize its expected cumulative revenue. Without relying on restrictive assumptions, the DD-DRO framework provides a pricing strategy with both expected revenue guarantees and tractability.

Specifically, in this setting, under \cref{assumption:continuous}, given $\hPN(\xb)$ in \cref{eq:def-hpn} and the objective function $R(\cdot,\cdot)$ in \cref{eq:pricing-obj}, the problem \eqref{DD-DRO} can be reformulated as the following optimization problem with finitely many convex constraints:
\begin{align*}
    \max_{\xb,\nu,w,\lambda_i,i\in[N_\xi]}&\ \sum_{i=1}^{N_\xi}f_i(\xb)\nu_i+r_N\cdot\nu_{N_\xi+1} \\
    \mathrm{s.t.}&\ -\xi_U\mathbf{1}_{T}^\mathrm{T}\lambda_i-w^\mathrm{T}\xi_{m_i}\ge \nu_i,\ \forall i\in[N_\xi],\\
    &\ w\in\R^{T},||w||_q\le -\nu_{N_{\xi}+1}, \\
    &\ \lambda_i\le \xb+w, \\
    &\ \nu_1,\cdots,\nu_{N_\xi}\in\R,\ \nu_{N_\xi+1}\le 0,\\
    &\ \lambda_i\in\R^{T},\lambda_i\ge 0,\ \forall i\in[N_\xi],
\end{align*}
where $||\cdot||_q$ is the dual norm of $||\cdot||_p$ satisfying $\frac{1}{p}+\frac{1}{q}=1$, and $\mathbf{1}_T$ is a $T$-dimensional column vector of ones. 
Furthermore, given $r_N$ in \cref{eq:def-rn}, the expected revenue can be bounded as described in \cref{thm:out-of-sample-performance}, with $c_p = x_U$.

We empirically evaluate our proposed approach to validate the theoretical results. In the experiment, we model demand as following a time-varying normal distribution, where its mean vector $[\mu_t(\xb)]$ decreases with respect to the time-averaged price:
\begin{equation}
    \mu_t(\xb)=(1.4-0.2\cdot t)\cdot (1.7\cdot x_U-\frac{\sum_{t=1}^Tx_t}{T}),
\end{equation}
and the covariance matrix is a identity matrix. We set $x_U = 1$, $\xi_U = 5$, and $T = 3$. We adopt nearest-neighbor interpolation and vary both the ambiguity radius $r_N$ and the sample size. We then compare the values of $\hat{J}_N$, $J_N$, and $J^\star$. The results, averaged over 10 random seeds, are shown in \cref{fig:exp_tmp}.
\begin{figure}
    \centering
    \begin{subfigure}[b]{0.48\textwidth}
    \includegraphics[width=\linewidth]{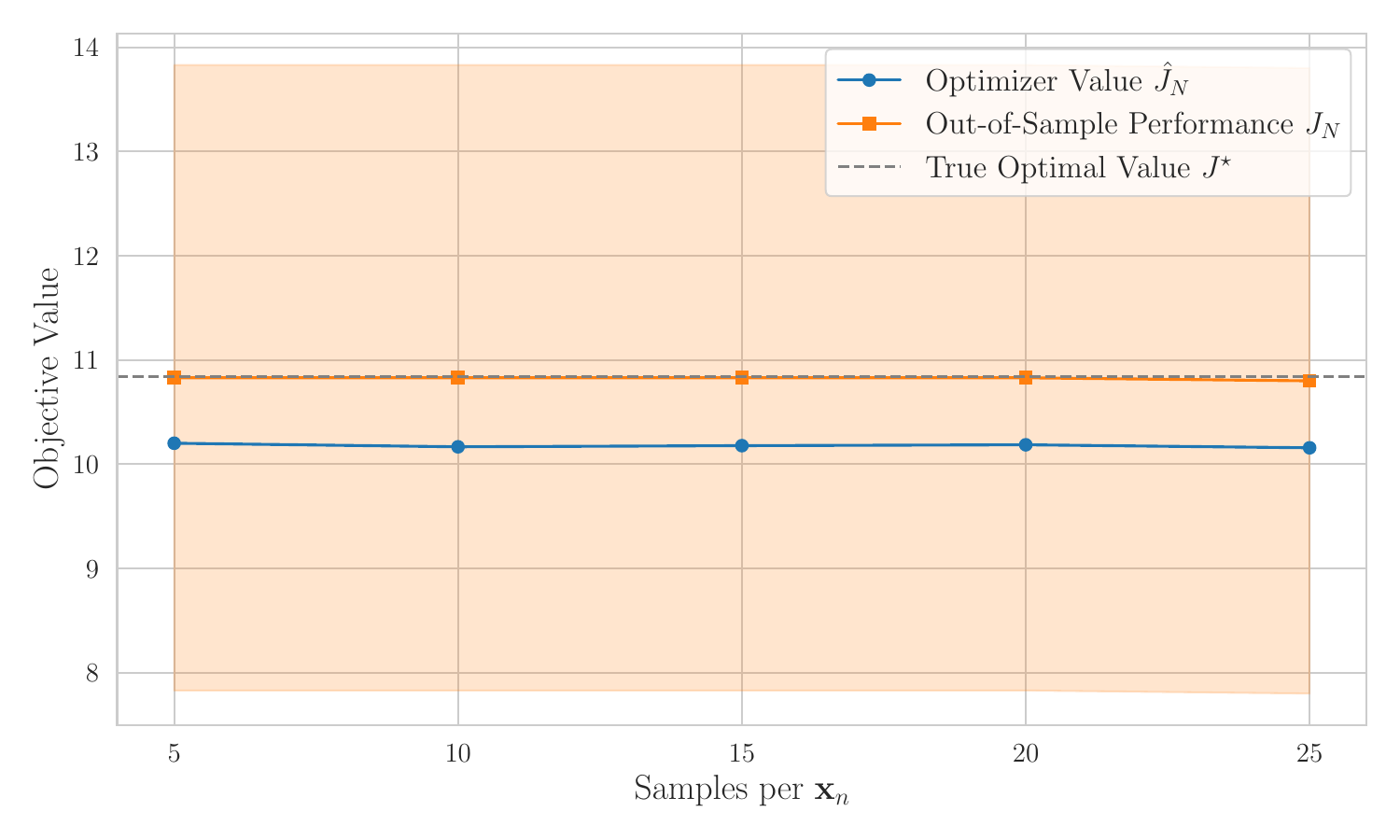}
    \caption{The comparison of objective values as we vary the sample size. $r_N$ is set to be $1.5$.}\label{fig:exp1}
    \end{subfigure}
    \begin{subfigure}[b]{0.48\textwidth}
    \includegraphics[width=\linewidth]{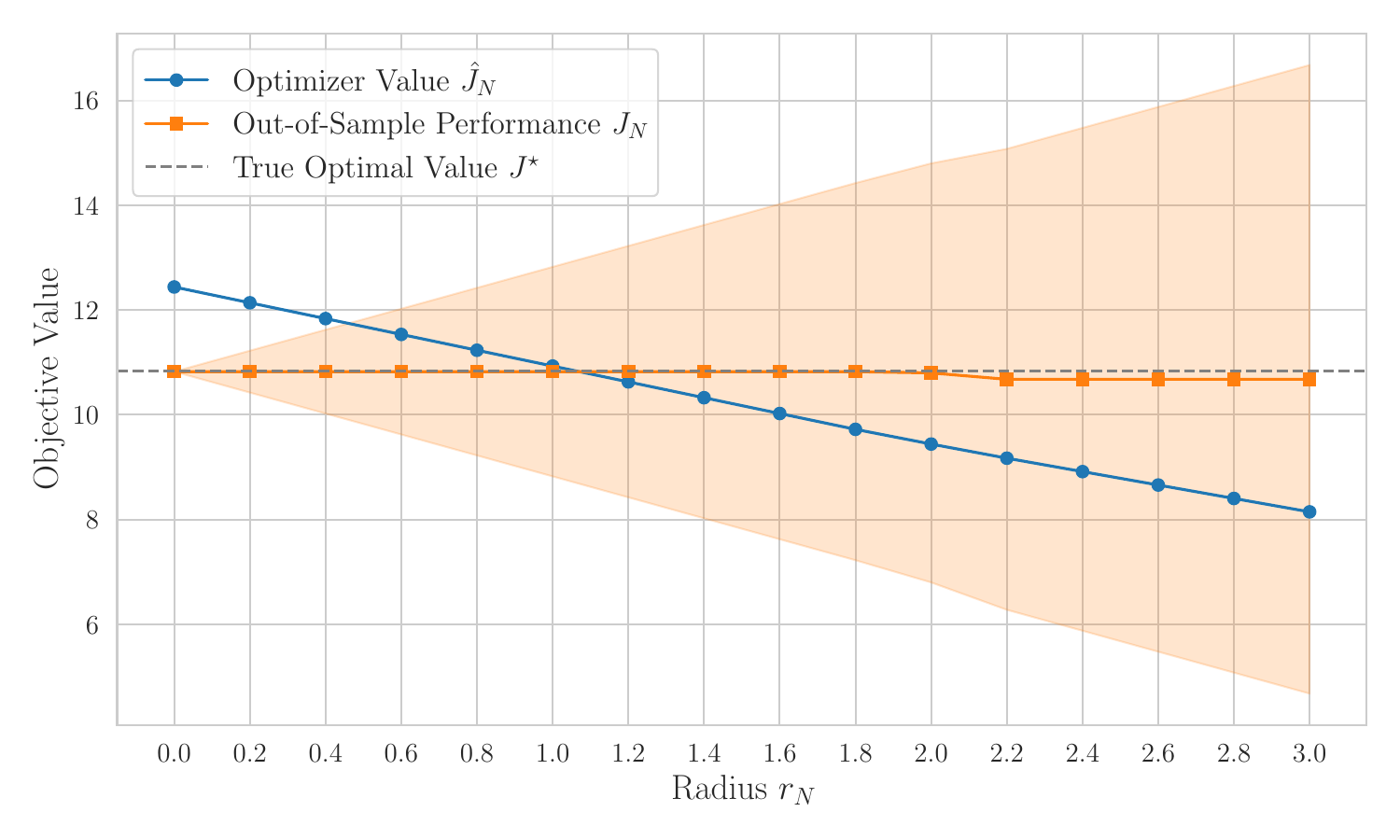}
    \caption{The comparison of objective values as we vary $r_N$. Sample size is set to be $15$ at each $\xb_n$.}\label{fig:exp2}
    \end{subfigure}
    \caption{Experiment results for the dynamic pricing problem. The shaded area is $[J_N- 2x_U\cdot r_N,J_N+2x_U\cdot r_N]$.}
    \label{fig:exp_tmp}
\end{figure}

As shown in \cref{fig:exp1}, with a properly chosen $r_N$, the result supports the validity of \cref{thm:out-of-sample-performance}, and the out-of-sample performance $J_N$ is comparable to the optimal value $J^\star$ and is relatively stable with the sample size. (Note that in the pricing problem, we solve a max–min formulation, so the inequality directions in \cref{thm:out-of-sample-performance} are reversed accordingly.)

In \cref{fig:exp2}, we demonstrate that \cref{thm:out-of-sample-performance} holds across a wide range of $r_N$ values, and that the out-of-sample performance $J_N$ remains stable without significant degradation as $r_N$ increases.
\section{conclusion}
This paper introduces a novel decision-dependent distributionally robust optimization (DD-DRO) framework, targeting for settings in which the underlying model parameter distributions are both unknown and impacted by the decisions. By constructing ambiguity sets via interpolation techniques combined with the Wasserstein metric, we show that this DD-DRO framework is both theoretically grounded and computationally tractable. 

Beyond the technical contributions, this work presents a conceptual shift: instead of globally modeling or learning the data-generating distribution, the decision-makers can interpolate local empirical behavior by controlling robustness level via principled, geometry-aware metrics. This perspective is particularly valuable in data-rich but structure-poor environments, where parametric modeling can be unreliable or infeasible.


\bibliographystyle{IEEEtran}
\bibliography{root}




\section*{APPENDIX}
\subsection{Proof of \cref{thm:tractable-reformulation}}\label{appendix:tractable-reformulation}
\begin{proof}
    For $i\in N_\xi$, let $\Xi^i\triangleq\{\xi_{m_i}\}$, and $\Xi^{N_\xi+1}\triangleq\Xi \backslash(\cup_{i=1}^{N_\xi}\Xi^i)$. By \cite{LUO201920}[Theorem 3.6], the inner problem of \eqref{DD-DRO} is equivalent to the following conic linear program:
    \begin{align*}
        \sup_{\mu}&\ \int_{(s_1,s_2)\in\Xi\times\Xi}h(\xb,s_1)\mu(ds_1\times ds_2)\\
        \mathrm{s.t.}&\ \int\ind[s_2\in\Xi^i]\mu(ds_1\times ds_2)=f_i(\xb),\ i\in[N_\xi],\\
        &\ \int\ind[s_2\in\Xi^{N_\xi+1}]\mu(ds_1\times ds_2)=0, \\
        &\ \int\bigg[\sum_{i\in[N_\xi]}d(s_1,s_2)\ind[s_2\in\Xi^i]\bigg]\mu(ds_1\times ds_2)\le r_N, \\
        &\ \mu \succeq 0.
    \end{align*}
    Note that both the objective function and constraints are linear in $\mu$. To express this conic linear program using operator notation, we define a family of functions $\{h_i\}_{0\le i\le N_\xi+2}$ on $\Xi\times\Xi$ as follows:
    \begin{equation*}
        h_i(s_1,s_2)\triangleq\left\{
        \begin{aligned}
            &h(\xb,s_1),\ i=0,\\
            &\ind[s_2\in \Xi^i],\ i\in[N_\xi+1],\\
            &\sum_{j\in[N_\xi]}d(s_1,s_2)\ind[s_2\in\Xi^j],\ i=N_\xi+2.
        \end{aligned}
        \right.
    \end{equation*}
    Each $h_i$ is $\cB(\Xi) \times \cB(\Xi)$-measurable and bounded by \cref{assumption:compact-continuous}, and hence integrable with respect to any $\mu \succeq 0$. Let $\cM$ denote the linear space of finite signed measures, and define $\cM_+\triangleq\{\mu\in\cM\mid\mu\succeq 0\}$ as the cone of nonnegative measures. Let $\cM_+'$ be the space of all functions integrable with respect to any $\mu\in\cM_+$. Define the bilinear form $\langle\mu,g\rangle\triangleq\int_{(s_1,s_2)\in\Xi\times\Xi }g(s_1,s_2)\mu(ds_1\times ds_2)$ for $\mu \in \cM_+$ and $g \in \cM_+'$. For $\mu\in\cM_+$, we further define the linear operator $\cA:\cM_+\mapsto \R^{N_\xi+2}$ as:
    \begin{equation*}
        \cA\mu\triangleq\bigg[\langle\mu,h_1\rangle,\cdots,\langle\mu,h_{N_\xi+2}\rangle\bigg]^{\top},
    \end{equation*}
    where the superscript $\top$ denotes the transpose. Let
    \begin{equation*}
        \bb(\xb)\triangleq[f_1(\xb),\cdots,f_{N_\xi}(\xb),0,r_N]^{\mathrm{T}}\in \R^{N_\xi+2},
    \end{equation*}
    and define the convex cone $\cK\triangleq\mathbf{0}^{N_\xi+1}\times(-\infty,0]\subset\R^{N_\xi+2}$. The primal problem becomes:
    \begin{align*}
        \sup_{\mu}&\ \langle\mu,h_0\rangle\\
        \mathrm{s.t.}&\ \cA\mu-\bb(\xb)\in\cK,\\
        &\ \mu\in\cM_+.
    \end{align*}
This is a standard conic linear program. Its dual is given by:
\begin{align*}
        \inf_{ w}&\ -\bb(\xb)\cdot  w\\
        \mathrm{s.t.}&\ \cA^\star w+h_0\in-(\cM_+^\star),\\
        &\  w\in\cK^\star,
\end{align*}where $\cK^\star=\R^{N_\xi+1}\times
(-\infty,0]$ is the dual cone of $\cK$, and $-(\cM_+)^\star=\{h\in\cM_+':\langle\mu,h\rangle\le 0,\forall\mu\in\cM_+\}$ is the polar cone of $\cM_+$. For $w\in\R^{N_\xi+2}$, the adjoint operator $\cA^\star:\R^{N_{\xi}+2}\mapsto\R$ is defined by:
\begin{equation*}
    \cA^\star  w=\sum_{i\in[N_\xi+2]} w_ih_i.
\end{equation*}
Expanding the dual problem, we have:
\begin{align*}
    \inf_{ w}&\ -\sum_{i\in[N_\xi]}f_i(\xb) w_i-r_N\cdot w_{N_\xi+2} \\
    \mathrm{s.t.}&\ \langle\mu,h_0\rangle+\sum_{i\in[N_\xi+2]} w_i\langle\mu,h_i\rangle\ge 0,\ \forall\mu\in\cM_+,\\
    &\  w_{N_\xi+2}\ge 0.
\end{align*}
We expand the constraint for all $\mu$ using the fact that $\mu$ is nonnegative and arbitrary. This is equivalent to requiring the integrands to be nonnegative:
\begin{align*}
     w_i+ w_{N_\xi+2}\cdot d(s,\xi_{m_i})+h(\xb,s)&\ge 0,\ \forall s\in\Xi, i\in[N_\xi],\\
     w_{N_\xi+1}+h(\xb,s)&\ge 0,\ \forall s\in\Xi. 
\end{align*}
The second constraint can be removed because $w_{N_\xi + 1}$ does not appear in the objective and can be chosen large enough to satisfy the constraint. Substituting $\nu_i = -w_i$ for $i \in [N_\xi]$ and $\nu_{N_\xi + 1} = -w_{N_\xi + 2}$, we obtain the equivalent reformulation in \eqref{DD-DRO}.

Next, we verify that strong duality holds for any $\xb\in\cX$. By \cite{Shapiro2001}[Proposition 2.8(iii)], it suffices to show that the primal value is finite and that the dual has a non-empty and finite solution set.

By \cref{assumption:compact-continuous}, the objective $g(\xb)$ of the primal is finite for any $\xb\in\cX$. That is, there exist constants $\gamma_1$ and $\gamma_2$ such that:
\begin{align*}
    &\inf_{\xb\in\cX}g(\xb)\ge \gamma_1,\\
    &\sup_{\xb\in\cX}g(\xb)\le \gamma_2.
\end{align*}
One can verify that $\nu=[\gamma_2,\cdots,\gamma_2,0]^\top$ is a feasible solution. Furthermore, by evaluating the constraint at  $s=\xi_{m_i}$, we get:
\begin{equation*}
    \nu_i\ge h(\xb,\xi_{m_i})\ge \gamma_1, \quad \forall i\in[N_\xi].
\end{equation*}
Next, since $\sum_{i=1}^{N_\xi}f_i(\xb)\nu_i+r_N\cdot\nu_{N_\xi+1}\le \gamma_2$, we have $\nu_{N_\xi+1}\le \frac{1}{r_N}(\gamma_2-\gamma_1)$. The rest we need to do is to upper bound $\nu_i$. For a given $\xb\in\cX$, if there exists $i\in[N_\xi]$ such that $f_i(\xb)=0$, notice that $\nu_i$ will not appear in the objective function, and can be arbitrarily large to make the constraint that involves $\nu_i$ is satisfied. Therefore, the existence of $\nu_i$ does not affect the optimal value of the dual problem, and thus can be removed. Therefore, we focus on cases where $f_i(\xb)>0$. Let $\gamma_3\triangleq\inf_{i\in[N_\xi],\xb\in\cX}f_i(\xb)$. By \cref{assumption:compact-continuous}, $\gamma_3>0$. By \cref{def:interpolation}, we have $\sum f_i(\xb)= 1$. Therefore, since $\sum_{i=1}^{N_\xi}f_i(\xb)\nu_i+r_N\cdot\nu_{N_\xi+1}\le \gamma_2$, we have
\begin{equation*}
    \nu_i\le \frac{1}{\gamma_3}(\gamma_2-r_N\cdot\nu_{n_\xi+1}-\sum_{j\neq i}f_j(\xb)\nu_j)\le \frac{1}{\gamma_3}(2\gamma_2-\gamma_1).
\end{equation*}
Therefore, by \cite{Shapiro2001}[Proposition 2.8(iii)], strong duality holds.
\end{proof}

\end{document}